\documentclass[11pt]{amsart}
\usepackage{amsmath}
\usepackage{amsfonts}

\usepackage{amscd}
\usepackage{amsthm}
\usepackage{amssymb} \usepackage{latexsym}
\usepackage{eufrak}
\usepackage{euscript}
\usepackage{epsfig}
\usepackage{graphics}
\usepackage{array}
\usepackage{enumerate}
\usepackage{dsfont}
\usepackage{color}
\usepackage{wasysym}
\usepackage{hyperref}
\usepackage{pdfsync}

\newcommand{\bel}[1]{\begin{equation}\label{#1}}

\newcommand{\be}{\begin{equation}}

\newcommand{\ba}{\begin{eqnarray}}
\newcommand{\ea}{\end{eqnarray}}

\newcommand{\qe}{\end{equation}}
\newcommand{\R}{{\mathbb R}}

\newcommand{\vol}{{\mathrm{vol}}}

\newcommand{\CCD}{\mathrm{CD}}

\newcommand{\Ric}{\mathrm{Ric}}

\newcommand{\Hmm}[1]{\leavevmode{\marginpar{\tiny%
$\hbox to 0mm{\hspace*{-0.5mm}$\leftarrow$\hss}%
\vcenter{\vrule depth 0.1mm height 0.1mm width \the\marginparwidth}%
\hbox to
0mm{\hss$\rightarrow$\hspace*{-0.5mm}}$\\\relax\raggedright #1}}}

\theoremstyle{theorem}
\newtheorem{thm}{Theorem}

\theoremstyle{example}

\theoremstyle{conj}

\theoremstyle{corollary}

\theoremstyle{lemma}
\newtheorem{lemma}[thm]{Lemma}
\theoremstyle{definition}

\theoremstyle{proof}
\theoremstyle{remark}

\begin{document}

\title[Liouville theorem for bounded harmonic functions]{Liouville theorem for bounded harmonic functions on manifolds and graphs satisfying non-negative curvature dimension condition}
\author{Bobo Hua}
\email{bobohua@fudan.edu.cn}
\address{School of Mathematical Sciences, LMNS, Fudan University, Shanghai 200433, China; Shanghai Center for Mathematical Sciences, Fudan University, Shanghai 200433, China
}

\begin{abstract} Brighton \cite{Brighton13} proved the Liouville theorem for bounded harmonic functions on weighted manifolds satisfying non-negative curvature dimension condition, i.e. $\CCD(0,\infty).$ In this paper, we provide a new proof of this result by using the reverse Poincar\'e inequality. Moreover, we adopt this approach to prove the Liouville theorem for bounded harmonic functions on graphs satisfying the $\CCD(0,\infty)$ condition.
\end{abstract}
\maketitle

\section{Introduction}\label{s:introduction}
Yau \cite{Yau75} proved that positive harmonic functions are constant on a complete, noncompact Riemannian manifold with non-negative Ricci curvature. As a corollary, any bounded harmonic function is constant. These are called Liouville theorems for harmonic functions, regarded as the generalizations of classical Liouville's theorem for bounded holomorphic functions on the complex plane. Later, Cheng and Yau \cite{ChengYau75} obtained a related quantitative result, so-called Cheng-Yau gradient estimate for positive harmonic functions. This type of gradient estimate was generalized to positive solutions of heat equations by Li and Yau \cite{LiYau86}.
The space of positive (bounded resp.) harmonic functions corresponds to the Martin (Poisson resp.) boundary at infinity for Brownian motion or random walks in the probability theory, see e.g. \cite{Kaimanovich96,Woess09}. In this terminology, the Liouville theorem means the triviality of the corresponding boundary at infinity. Liouville theorems for harmonic functions have received much attention in the literature, to cite a few \cite{KaimanovichVershik83,Sullivan83,Anderson83,LiTam87,Lyons87,Grigoryan90,Benjamini91,Grigoryan91,Saloff92,WangFY02,Erschler04,Li12,Brighton13}.
In this paper, we study Liouville theorems for bounded harmonic functions on weighted manifolds and graphs.

We first recall some facts on Riemannian manifolds.
Let $(M,g)$ be a $d$-dimensional complete Riemannian manifold with the Riemannian metric $g$, $\Delta$ be the Laplace-Beltrami operator on $M$ and $C_0^\infty(M)$ be the space of compactly supported smooth functions on $M.$ 
As is well-known, the Ricci curvature tensor of a manifold is bounded below by $Kg$ and the dimension $d$ is at most $n,$ $K\in \R$ and $n\in(0,\infty],$ if and only if the following Bochner inequality holds,
\begin{equation}\label{eq:eqb}\frac12\Delta |\nabla f|^2\geq \frac{1}{n}(\Delta f)^2+\langle \nabla f,\nabla \Delta f\rangle+K|\nabla f|^2,\quad\forall \ f\in C_0^\infty(M),\end{equation} where $\nabla\cdot$ is the Riemannian gradient of a function. In a  general setting, we consider a Riemannian manifold $(M,g)$ equipped with a measure $e^{-V(x)}d\vol(x),$ where $V$ is a smooth function on $M,$ and call the triple $(M,g,e^{-V}d\vol)$ a weighted manifold. Weighted manifolds, also called smooth metric measure spaces, have been extensively studied in the literature, see e.g. \cite{Morgan05,WeiW09,Grigoryan09book,MunteanuW11,MunteanuW12}. The associated $V$-Laplacian is defined as
$$\Delta_V f= \Delta f- \langle\nabla V, \nabla f \rangle, \quad\forall \ f\in C_0^\infty(M).$$ A function $u$ is called $V$-harmonic if $\Delta_V u=0$ on $M.$ The weighted Ricci curvature is defined as 
$$\Ric_V=\Ric +\nabla^2 V,$$ where $\nabla^2\cdot$ is the Hessian of a function.
Yau's Liouville theorem for bounded harmonic functions was generalized to weighted manifolds by Brighton \cite{Brighton13}.
\begin{thm}[\cite{Brighton13}]\label{thm:brighton} Let $(M,g,e^{-V}d\vol)$ be a complete weighted manifold satisfying $\Ric_V\geq 0.$ Then any bounded $V$-harmonic function is constant.
\end{thm}
Brighton adopted the techniques for Cheng-Yau's gradient estimate \cite{ChengYau75} to establish the following:
For any $q\in M,$ $R\geq 1$ and any positive $V$-harmonic function $u$ on $B_R(q),$ the ball of radius $R$ centered at $q$, we have 
$$|\nabla u|(q)\leq \frac{C}{\sqrt{R}}\sup_{B_{R}(q)}u,$$ where $C$ is a constant independent of $R.$ The theorem follows from passing to the limit $R\to \infty.$

In this paper, we will give a new proof of Theorem~\ref{thm:brighton} using semigroup approaches.
For general Markov semigroups, Bakry and \'Emery \cite{BakryEmery85,Bakry87,BakryGentilLedoux} introduced the so-called $\Gamma$-calculus, and defined the curvature dimension condition mimicking \eqref{eq:eqb}, denoted by $\CCD(K,n)$, and derived various interesting analytic properties of the semigroups under these conditions. For a weighted manifold $(M,g,e^{-V}d\vol),$ we denote by $P_t=e^{t\Delta_V}$ the associated heat semigroup. In this setting, the $\CCD(0,\infty)$ condition means that $\Ric_V\geq 0.$ It is well-known \cite{BakryGentilLedoux,Wang10} that $\Ric_V\geq 0$ is equivalent to
\begin{equation}\label{eq:gradient}|\nabla P_t f|^2\leq P_t|\nabla f|^2, \quad\forall \ f\in C_0^\infty(M), t>0,\end{equation} or
\begin{equation}\label{eq:reverse Poincare}
2t |\nabla P_t f|^2\leq P_t(f^2)-(P_tf)^2, \quad\forall \ f\in C_0^\infty(M), t>0,\end{equation} where the latter is called the reverse Poincar\'e inequality. We will use the reverse Poincar\'e inequality to give a proof of Theorem~\ref{thm:brighton} in Section~\ref{sec:Brighton}. The {\textbf{heuristic argument}} is as follows: Let $f$ be a bounded harmonic function on $M.$ We may apply \eqref{eq:reverse Poincare} to $f,$ to be justified by proper approximation. By the semigroup property and the boundedness of $f$, the right hand side of \eqref{eq:reverse Poincare} is bounded above by a constant $C$ independent of $t.$ Moreover, by the harmonicity of $f,$ $P_t f=f,$ which implies that
\begin{equation}\label{eq:argument}|\nabla f|^2=|\nabla P_t f|^2\leq \frac{C}{2t}\to 0,\quad t\to \infty.\end{equation} This yields that $|\nabla f|\equiv 0$ and hence $f$ is constant.

In the following we introduce the setting of graphs, see \cite{HuaLin17} for notation, and adapt the previous argument to the discrete case.
Let $(V,E)$ be a connected, undirected, combinatorial graph with the set of vertices $V$ and the set of edges $E.$  We say $x,y\in V$ are neighbors, denoted by $x\sim y,$ if there is an edge connecting $x$ and $y.$ The graph is called \emph{locally finite} if each vertex has finitely many neighbors. In this paper, all graphs we consider are locally finite. We assign a weight $m$ to each vertex, $m: V\to (0,\infty),$ and a weight $\mu$ to each edge, $$\mu:E\to (0,\infty),\ E\ni \{x,y\}\mapsto \mu_{xy}=\mu_{yx},$$ and refer to the quadruple $G=(V,E,m,\mu)$ as a \emph{weighted graph}. We denote by $C_0(V)$ the set of finitely supported functions on $V,$ by $\ell^p(V,m),$ $p\in [1,\infty],$ the $\ell^p$ spaces of functions on $V$ with respect to the measure $m,$ and by $\|\cdot\|_{\ell^p(V,m)}$ the $\ell^p$ norm of a function.

For any weighted graph $G=(V,E,m,\mu),$ we define the \emph{Laplacian} $\Delta$ as
$$\Delta f(x)=\frac{1}{m(x)}\sum_{y\in V:y\sim x}\mu_{xy}(f(y)-f(x)),\quad \forall\ f:V\to\R.$$ It is proven by Keller and Lenz \cite{KellerLenz12} that $\Delta$ is a bounded linear operator on $\ell^2(V,m)$ if and only if \begin{equation}\label{eq:bounded}\sup_{x\in V}\frac{1}{m(x)}\sum_{y\in V:y\sim x}\mu_{xy}<\infty.\end{equation} From another point of view, for a weighted graph $G$, it associates with a Dirichlet form with respect to the Hilbert space $\ell^2(V,m),$
\begin{eqnarray*}\label{e:dirichlet form}Q:&&D(Q)\times D(Q)\to \R \nonumber\\ && (f,g)\mapsto \frac12\sum_{x,y\in V:x\sim y}\mu_{xy}(f(y)-f(x))(g(y)-g(x)),\end{eqnarray*} where the form domain $D(Q)$ is defined as the completion of $C_0(V)$ under the norm $\|\cdot\|_Q$ given by $$\|f\|^2_{Q}=\|f\|_{\ell^2(V,m)}^2+\frac12\sum_{x,y\in V:x\sim y}\mu_{xy}(f(y)-f(x))^2, \ \ \forall f\in C_0(V),$$ see \cite{KellerLenz12}. We denote by $\mathcal{L}$ the infinitesimal generator of the Dirichlet form $Q$ and by $P_t=e^{t\mathcal{L}}$ the corresponding $C_0$-semigroup.
For a locally finite graph, the generator $\mathcal{L}$ coincides with the Laplacian $\Delta$ on the domain of generator $D(\mathcal{L})$ which contains $C_0(V),$ i.e. $\mathcal{L}f=\Delta f,$ for any $f\in D(\mathcal{L})$, see \cite[Theorem~6~and~9]{KellerLenz12}.

Note that the weights $\mu$ and $m$ determine the properties of the Laplacian and the semigroup. Given the edge weights $\mu,$ typical choices of $m$ are of particular interest: \begin{itemize}\item $m(x)=\sum_{y:y\sim x}\mu_{xy}$ for any $x\in V$ and the associated Laplacian is called the normalized Laplacian.
\item $m\equiv1$ and the Laplacian is called the physical (or combinatorial) Laplacian.
\end{itemize}

Now we introduce the $\Gamma$-calculus on graphs following \cite{BakryGentilLedoux,LinYau10}.
The ``carr\'e du champ" operator $\Gamma$ is defined as
$$\Gamma(f,g)(x):=\frac12(\Delta(fg)-f\Delta g-g\Delta f)(x), \quad \forall f,g:V\to\R, x\in V.$$ For simplicity, we write $\Gamma(f):=\Gamma(f,f).$ The iterated $\Gamma$ operator, $\Gamma_2$, is defined as
$$\Gamma_2(f):=\frac{1}{2}\Delta \Gamma(f)-\Gamma(f,\Delta f).$$

A Riemannian manifold $(M,g)$ is complete if and only if there exists a non-decreasing sequence of non-negative functions $\{\eta_k\}_{k=1}^\infty$ in $C_0^\infty(M)$ such that \begin{equation*}\lim_{k\to\infty}\eta_k=\mathds{1}\ \mathrm{and}\ \ |\nabla\eta_k|\leq \frac{1}{k},\end{equation*} where $\mathds{1}$ is the constant function $1$ on $M.$ Motivated by this, a weighted graph $G=(V,E,m,\mu)$ is called \emph{complete} if there exists a non-decreasing sequence of non-negative finitely supported functions $\{\eta_k\}_{k=1}^\infty$ such that \begin{equation}\label{d:complete}\lim_{k\to\infty}\eta_k=\mathds{1}\ \mathrm{and}\ \ \Gamma(\eta_k)\leq \frac{1}{k}.\end{equation} Note that  the weights $\mu$ and $m$ are crucial to the completeness of a weighted graph. This condition was defined for Markov diffusion semigroups in \cite[Definition~3.3.9]{BakryGentilLedoux} and adapted to graphs in \cite{HuaLin17}. A large class of graphs possessing appropriate intrinsic metrics have been shown to be complete, see \cite[Theorem~2.8]{HuaLin17}, where $\eta_k$ are constructed via distance functions to some vertices. In particular, graphs with bounded Laplacians are always complete.

In this paper, we say that a weighted graph $G$ satisfies the assumption $(A)$ if one of the following holds:
 \begin{enumerate}
  \item[$(A_1)$]\label{ass:A1} The Laplacian $\Delta$ is bounded on $\ell^2(V,m),$ i.e. \eqref{eq:bounded} holds.
  \item[$(A_2)$]\label{ass:A2} $G$ is complete and $m$ is non-degenerate, i.e. $\inf_{x\in V}m(x)>0.$
  \end{enumerate}
Note that normalized Laplacians are always bounded, hence satisfy $(A_1);$ physical Laplacians satisfy $(A_2)$ provided that they are complete.

For our purposes, we need the curvature dimension condition on graphs, which was initiated in \cite{BakryEmery85,Bakry87} on manifolds (or Markov diffusion semigroups) and introduced to graphs by \cite{Schmuck96,LinYau10}. We say that a graph $G=(V,E,m,\mu)$ satisfies the $\CCD(K,n)$ condition, for $K\in \R$ and $n\in (0,\infty],$ if
  $$\Gamma_2(f)(x)\geq \frac1n (\Delta f)^2(x)+K\Gamma (f)(x),\quad \forall\ f:V\to\R,\ x\in V.$$
Now we are ready to state our main result.
\begin{thm}\label{thm:main1} Let $G=(V,E,m,\mu)$ be a weighted graph satisfying the assumption $(A)$ and the $\CCD(0,\infty)$ condition. Then any bounded harmonic function on $G$ is constant.
\end{thm}

We remark that the Liouville theorem for bounded harmonic functions on graphs can be derived from the Li-Yau gradient estimate or the Harnack inequality. They had been obtained by \cite{BHLLMY15,HornLinLiuYau14} under other curvature dimension conditions, $\mathrm{CDE}(0,n)$ and $\mathrm{CDE'}(0,n),$ for bounded Laplacians and finite $n.$ It was proven by M\"unch \cite{Munch17} that $\mathrm{CDE'}(0,n)$ implies that $\CCD(0,n).$ So that our assumption of $\CCD(0,\infty)$ is much weaker. The assumption $(A)$ is essential for the following result, which was proved for finite graphs by \cite{LiuP14}, for graphs with bounded Laplacians by \cite{LinLiu15}, and partially for graphs with unbounded Laplacians satisfying $(A_2)$ by \cite{HuaLin17} and \cite{GongLin15}.
\begin{thm}\label{thm:main2} For a weighted graph $G=(V,E,m,\mu)$ satisfying the assumption $(A),$ the following are equivalent:
\begin{enumerate}[(a)]
   \item $G$ satisfies $\CCD(0,\infty).$
    \item
    $\Gamma(P_t f)\leq P_t(\Gamma (f)),\quad\quad\quad\quad\quad\  \forall\ f\in C_0(V), t>0.$
    \item
    $P_t(f^2)-(P_t f)^2\leq 2t P_t(\Gamma(f)),\quad \forall\ f\in C_0(V), t>0.$
    \item
    $2t \Gamma(P_tf)\leq P_t(f^2)-(P_t f)^2,\quad\ \  \forall\ f\in C_0(V), t>0.$
\end{enumerate}
\end{thm}


The paper is organized as follows: We prove Theorem~\ref{thm:brighton} in the next section. The last section is devoted to the proofs of Theorem~\ref{thm:main1} and Theorem~\ref{thm:main2}.

\section{Proof of Theorem~\ref{thm:brighton}}\label{sec:Brighton}
In this section, we prove Theorem~\ref{thm:brighton}.

\begin{lemma}\label{lem:ptf}  Let $(M,g,e^{-V}d\vol)$ be a complete weighted manifold satisfying $\Ric_V\geq 0,$ $f$ be a bounded $V$-harmonic function. Then $P_tf=f.$
\end{lemma}
\begin{proof}By the gradient estimate \eqref{eq:gradient}, the semigroup $P_t$ is stochastically complete, see e.g. \cite[Theorem~3.3.23 and Corollary~3.2.2]{BakryGentilLedoux}. So that any bounded solution to the heat equation $\partial_t u=\Delta u$ is unique, see \cite[Theorem 8.18]{Grigoryan09book}.
Note that by the harmonicity of $f,$ $u(t,x)=f(x)$ solves the heat equation with initial data $f$ which is bounded for all $t\geq 0.$ Hence $P_t f=u=f.$ 
\end{proof}

\begin{proof}[Proof of Theorem~\ref{thm:brighton}]
Let $(M,g,e^{-V}d\vol)$ be a complete weighted manifold satisfying $\Ric_V\geq 0,$ $f$ be a bounded $V$-harmonic function. Let $\{\eta_k\}_{k=1}^\infty$ be a non-decreasing sequence of non-negative compactly supported  smooth functions such that $$\lim_{k\to\infty}\eta_k=\mathds{1}.$$ Set $f_k:=f\cdot \eta_k,$ for $k\geq1.$ By the dominated convergence theorem, one can show that for any $t>0,$
$$P_t f_k\to P_t f, \quad k\to \infty.$$ By the reverse Poincar\'e inequality, \eqref{eq:reverse Poincare},
$$2t|\nabla P_tf_k|^2\leq P_t(f_k^2)-(P_t f_k)^2\leq C^2,$$ where $C=\|f\|_{\infty}.$ Hence
\begin{equation}\label{eq:part1}|\nabla P_t f_k|\leq \frac{C}{\sqrt{2t}}.\end{equation} For any fixed $x,y\in M,$ let $\gamma_{xy}(s),$ $s\in [0,d(x,y)]$, be the shortest geodesic connecting $x$ and $y$ parametrized by the arc length. Then by \eqref{eq:part1},
\begin{eqnarray*}|P_t f_k(y)-P_t f_k(x)|&\leq &\int_0^{d(x,y)}\left|\frac{d}{ds} P_tf_k(\gamma_{xy}(s))\right|ds\\
&\leq& d(x,y)\frac{C}{\sqrt{2t}}.
\end{eqnarray*} Passing to the limit $k\to\infty,$ and applying Lemma~\ref{lem:ptf}, we have
$$|f(y)-f(x)|=|P_tf(y)-P_tf(x)|\leq d(x,y)\frac{C}{\sqrt{2t}}.$$
Moreover, letting $t\to\infty,$ we conclude that $f(x)=f(y)$ for any $x,y\in M.$ This proves the theorem.

\end{proof}

\section{Proof of Theorem~\ref{thm:main1}}
In this section, we will prove Theorem~\ref{thm:main1}. For that purpose, we first prove Theorem~\ref{thm:main2}.

\begin{proof}[Proof of Theorem~\ref{thm:main2}] For bounded Laplacians, the theorem has been proved in \cite[Theorem~3.1]{LinLiu15}. It suffices to prove the result under the assumption of $(A_2).$
The equivalence of $(a)$ and $(b)$ was obtained in \cite[Theorem~1.1]{HuaLin17}. The equivalence of $(a)$ and $(c)$ was proved in \cite[Theorem~3.2]{GongLin15}. 

We prove $(a)\Longrightarrow(d)$ following \cite[Theorem~1.1]{Wang10} or \cite[Theorem~3.1]{LinLiu15}. The reverse direction is omitted here since it follows verbatim. Let $f\in C_0(V).$
We claim that $(P_tf)^2\in D(\mathcal{L})$ for any $t\geq 0.$ By the assumption $(A_2),$ \cite[Theorem~5]{KellerLenz12} yields that $$D(\mathcal{L})=\{u\in \ell^2(V,m):\Delta u\in \ell^2(V,m)\}.$$ On one hand, one can see that $(P_t f)^2\in \ell^2(V,m)$ since $$\|P_tf\|_{\ell^p(V,m)}\leq \|f\|_{\ell^p(V,m)},\quad \forall\ p\in [1,\infty].$$ On the other hand, $$\Delta (P_t f)^2=2P_t f\Delta P_t f+2\Gamma(P_t f).$$ Note that $\Delta P_t f\in \ell^2(V,m)$ since $P_t f\in D(\mathcal{L}).$ Moreover, by \cite[Proposition~3.7]{HuaLin17}, $$\|\Gamma(P_t f)\|_{\ell^1(V,m)}\leq \|\Gamma(f)\|_{\ell^1(V,m)}<\infty.$$ Hence $\Gamma(P_t f)\in \ell^2(V,m)$ which follows from $\ell^1(V,m)\subset \ell^2(V,m),$ derived from the assumption $(A_2).$ Combining these facts,  we have $\Delta(P_t f)^2\in \ell^2(V,m)$ and prove the claim.

Given $t>0,$ set $\phi(s,x)=P_s(P_{t-s}f)^2(x),\ 0\leq s\leq t.$ Direct calculation shows that
\begin{eqnarray*}\frac{d}{ds}\phi&=&\Delta P_s(P_{t-s}f)^2-P_s(2P_{t-s}f\Delta P_{t-s}f)\\
&=&P_s\left[\Delta (P_{t-s}f)^2-2P_{t-s}f\Delta P_{t-s}f\right]=2P_{s}\Gamma(P_{t-s}f),
\end{eqnarray*} where the second equality follows from that $\Delta P_s(P_{t-s}f)^2=P_s \Delta (P_{t-s}f)^2$ since $(P_{t-s}f)^2\in D(\mathcal{L}),$ proved in the previous claim. Using $(b)$ in Theorem~\ref{thm:main2} with the approximation of $P_{t-s}f$ by finitely supported functions,
$$\frac{d}{ds}\phi=2P_{s}\Gamma(P_{t-s}f)\geq 2\Gamma(P_sP_{t-s}f)=2\Gamma(P_t f).$$ Integration over $s$ on $[0,t]$ gives the result $(d)$. This proves the theorem.

\end{proof}

\begin{lemma}\label{lem:lem1}  Let $G=(V,E,m,\mu)$ be a weighted graph satisfying the assumption $(A)$ and the $\CCD(0,\infty)$ condition. Then for any bounded harmonic function $f,$ $P_tf=f.$
\end{lemma}
\begin{proof} By $(b)$ in Theorem~\ref{thm:main2}, the semigroup $P_t$ is stochastically complete, see \cite[Theorem~1.2]{HuaLin17}. So that any bounded solution to the heat equation $\partial_t u=\Delta u$ is unique, see \cite[Theorem~1.5.1]{Huang11} or \cite[Theorem~1]{KellerLenz12}.
The lemma follows from the same argument as in the proof of Lemma~\ref{lem:ptf}.
\end{proof}

Now we can prove the Liouville theorem for bounded harmonic functions for non-negatively curved graphs.
\begin{proof}[Proof of Theorem~\ref{thm:main1}] By the heuristic argument \eqref{eq:argument} in the introduction and Lemma~\ref{lem:lem1},
it suffices to prove that $(d)$ in Theorem~\ref{thm:main2} holds for any bounded function $f.$ By the completeness of the graph, there exists cut-off functions $\{\eta_k\}_{k=1}^\infty\subset C_0(V)$ satisfying \eqref{d:complete}. Set $f_k:=f\cdot \eta_k,$ for $k\geq1.$ Note that $f_k\to f,$ $k\to\infty,$ pointwise. By $(d)$ in Theorem~\ref{thm:main2},
\begin{equation}\label{eq:fork}2t \Gamma(P_tf_k)\leq P_t(f_k^2)-(P_t f_k)^2.\end{equation} Since $f$ is bounded, the dominated convergence theorem yields that $$P_t ((f_k)^a)\to P_t (f^a),\quad k\to\infty,\quad a=1,2.$$ Moreover, noting that the graph is locally finite, for any $x\in V,$ we have $$\Gamma(P_t(f_k))(x)\to \Gamma(P_tf)(x),\quad\ k\to\infty,$$ since both sides involve only finitely many summands. Passing to the limit, $k\to\infty,$ in \eqref{eq:fork}, we prove that $(d)$ in Theorem~\ref{thm:main2} holds for any bounded function $f.$ Hence the theorem follows.
\end{proof}

{\bf Acknowledgements.}
The author is supported by NSFC, grant no. 11401106.

\bibliography{Stoch2121}

\newcommand{\etalchar}[1]{$^{#1}$}
\begin{thebibliography}{BHL{\etalchar{+}}15}

\bibitem[And83]{Anderson83}
M.~Anderson.
\newblock {The Dirichlet problem at infinity for manifolds of negative
  curvature}.
\newblock {\em J. Differential Geom.}, 18(4):701--721, 1983.

\bibitem[Bak87]{Bakry87}
D.~Bakry.
\newblock {\'Etude des transformations de Riesz dans les vari\'et\'es
  Riemanniennes \`a courbure de Ricci minor\'ee}.
\newblock In {\em {S\'eminaire de Probabilit\'es, XXI}}, volume 1247 of {\em
  Lecture Notes in Math.}, pages 137--172, Berlin, 1987. Springer.

\bibitem[BE85]{BakryEmery85}
D.~Bakry and M.~\'Emery.
\newblock Diffusions hypercontractives.
\newblock In {\em S\'eminaire de Probabilit\'es, XIX (1983/84)}, volume 1123 of
  {\em Lecture Notes in Math.}, pages 177--206, Berlin, 1985. Springer.

\bibitem[Ben91]{Benjamini91}
I.~Benjamini.
\newblock {Instability of the Liouville property for quasi-isometric graphs and
  manifolds of polynomial volume growth}.
\newblock {\em J. Theoret. Probab.}, 4(3):631--637, 1991.

\bibitem[BGL14]{BakryGentilLedoux}
D.~Bakry, I.~Gentil, and M.~Ledoux.
\newblock {\em {Analysis and geometry of Markov diffusion operators}}.
\newblock Number 348 in Grundlehren der Mathematischen Wissenschaften.
  Springer, Cham, 2014.

\bibitem[BHL{\etalchar{+}}15]{BHLLMY15}
F.~Bauer, P.~Horn, Y.~Lin, G.~Lippner, D.~Mangoubi, and S.~T. Yau.
\newblock {Li-Yau inequality on graphs}.
\newblock {\em J. Differential Geom.}, 99(3):359--405, 2015.

\bibitem[Bri13]{Brighton13}
K.~Brighton.
\newblock {A Liouville-type theorem for smooth metric measure spaces}.
\newblock {\em J. Geom. Anal.}, 23(2):562--570, 2013.

\bibitem[CY75]{ChengYau75}
S.~Y. Cheng and S.~T. Yau.
\newblock {Differential equations on Riemannian manifolds and their geometric
  applications}.
\newblock {\em Comm. Pure Appl. Math.}, 28(3):333--354, 1975.

\bibitem[Ers04]{Erschler04}
A.~Erschler.
\newblock {Liouville property for groups and manifolds}.
\newblock {\em Invent. Math.}, 155(1):55--80, 2004.

\bibitem[GL17]{GongLin15}
Chao Gong and Yong Lin.
\newblock Equivalent properties for {CD} inequalities on graphs with unbounded
  {L}aplacians.
\newblock {\em Chin. Ann. Math. Ser. B}, 38(5):1059--1070, 2017.

\bibitem[Gri90]{Grigoryan90}
A.~Grigor'yan.
\newblock {Dimension of spaces of harmonic functions}.
\newblock {\em Math. Notes}, 48:1114--1118, 1990.

\bibitem[Gri91]{Grigoryan91}
A.~Grigor'yan.
\newblock {The heat equation on noncompact Riemannian manifolds (Russian)}.
\newblock {\em Mat. Sb.}, 182(1):55--87 (English translation in Math. USSR--Sb.
  72(1): 47--77, 1992.), 1991.

\bibitem[Gri09]{Grigoryan09book}
A.~Grigor'yan.
\newblock {\em Heat kernel and analysis on manifolds}, volume~47 of {\em AMS/IP
  Studies in Advanced Mathematics}.
\newblock American Mathematical Society, Providence, RI; International Press,
  Boston, MA, 2009.

\bibitem[HL17]{HuaLin17}
B.~Hua and Y.~Lin.
\newblock {Stochastic completeness for graphs with curvature dimension
  conditions}.
\newblock {\em Adv. Math.}, 306:279--302, 2017.

\bibitem[HLLY14]{HornLinLiuYau14}
P.~Horn, Y.~Lin, S.~Liu, and S.~T. Yau.
\newblock {Volume doubling, Poincar\'e inequality and Guassian heat kernel
  estimate for nonnegative curvature graphs}.
\newblock {\em arXiv:1411.5087}, 2014.

\bibitem[Hua11]{Huang11}
X.~Huang.
\newblock {On stochastic completeness of weighted graphs}.
\newblock {\em PhD thesis, Bielefeld University}, 2011.

\bibitem[Kai96]{Kaimanovich96}
V.~A. Kaimanovich.
\newblock {Boundaries of invariant Markov operators: the identification
  problem}.
\newblock In M.~Pollicott and K.~Schmidt, editors, {\em Ergodic theory of
  $\Z^d$ actions (Warwick, 1993--1994)}, London Math. Soc. Lecture Note Ser.
  228., pages 127--176, Cambridge, 1996. Cambridge Univ. Press.

\bibitem[KL12]{KellerLenz12}
M.~Keller and D.~Lenz.
\newblock {Dirichlet forms and stochastic completeness of graphs and
  subgraphs}.
\newblock {\em J. Reine Angew. Math.}, 666:189--223, 2012.

\bibitem[KV83]{KaimanovichVershik83}
V.~A. Kaimanovich and A.~M. Vershik.
\newblock {Random walks on discrete groups: boundary and Entropy}.
\newblock {\em Ann. Probab.}, 11(3):457--490, 1983.

\bibitem[Li12]{Li12}
P.~Li.
\newblock {\em {Geometric analysis}}, volume 134 of {\em Cambridge Studies in
  Advanced Mathematics}.
\newblock Cambridge University Press, Cambridge, 2012.

\bibitem[LL15]{LinLiu15}
Y.~Lin and S.~Liu.
\newblock {Equivalent properties of CD inequality on graph}.
\newblock {\em arXiv:1512.02677}, 2015.

\bibitem[LP14]{LiuP14}
S.~P. Liu and N.~Peyerimhoff.
\newblock {Eigenvalue ratios of nonnegatively curved graphs}.
\newblock {\em arXiv:1406.6617}, 2014.

\bibitem[LT87]{LiTam87}
P.~Li and L.~F. Tam.
\newblock {Positive harmonic functions on complete manifolds with nonnegative
  curvature outside a compact set}.
\newblock {\em Ann. of Math. (2)}, 125(1):171--207, 1987.

\bibitem[LY86]{LiYau86}
P.~Li and S.~T. Yau.
\newblock {On the parabolic kernel of the Schroedinger operator}.
\newblock {\em Acta Math.}, 156(3--4):153--201, 1986.

\bibitem[LY10]{LinYau10}
Y.~Lin and S.~T. Yau.
\newblock {Ricci curvature and eigenvalue estimate on locally finite graphs}.
\newblock {\em Math. Res. Lett.}, 17(2):343--356, 2010.

\bibitem[Lyo87]{Lyons87}
T.~Lyons.
\newblock {Instability of the Liouville property for quasi-isometric Riemannian
  manifolds and reversible Markov chains}.
\newblock {\em J. Differential Geom.}, 26(1):33--66, 1987.

\bibitem[M\"17]{Munch17}
F.~M\"unch.
\newblock {Remarks on curvature dimension conditions on graphs}.
\newblock {\em Calculus of Variations and Partial Differential Equations},
  11(1):11, 2017.

\bibitem[Mor05]{Morgan05}
F.~Morgan.
\newblock Manifolds with density.
\newblock {\em Notices Amer. Math. Soc.}, 52(8):853--858, 2005.

\bibitem[MW11]{MunteanuW11}
O.~Munteanu and J.~Wang.
\newblock Smooth metric measure spaces with non-negative curvature.
\newblock {\em Comm. Anal. Geom.}, 19(3):451--486, 2011.

\bibitem[MW12]{MunteanuW12}
O.~Munteanu and J.~Wang.
\newblock Analysis of weighted {L}aplacian and applications to {R}icci
  solitons.
\newblock {\em Comm. Anal. Geom.}, 20(1):55--94, 2012.

\bibitem[SC92]{Saloff92}
L.~Saloff-Coste.
\newblock {A note on Poincar\'e, Sobolev, and Harnack inequalities}.
\newblock {\em Internat. Math. Res. Notices}, (2):27--38, 1992.

\bibitem[Sch99]{Schmuck96}
M.~Schmuckenschl\"ager.
\newblock Curvature of nonlocal {M}arkov generators.
\newblock In {\em Convex geometric analysis ({B}erkeley, {CA}, 1996)},
  volume~34 of {\em Math. Sci. Res. Inst. Publ.}, pages 189--197. Cambridge
  Univ. Press, Cambridge, 1999.

\bibitem[Sul83]{Sullivan83}
D.~Sullivan.
\newblock {The Dirichlet problem at infinity for a negatively curved manifold}.
\newblock {\em J. Differential Geom.}, 18(4):723--732, 1983.

\bibitem[Wan02]{WangFY02}
F.~Y. Wang.
\newblock {Liouville theorem and coupling on negatively curved manifolds}.
\newblock {\em Stochastic Process. Appl.}, 100:27--39, 2002.

\bibitem[Wan11]{Wang10}
F.~Y. Wang.
\newblock {Equivalent semigroup properties for curvature-dimension condition}.
\newblock {\em Bull. Sci. Math.}, 135(6-7):803--815, 2011.

\bibitem[Woe09]{Woess09}
W.~Woess.
\newblock {\em Denumerable Markov chains--Generating functions, boundary
  theory, random walks on trees.}
\newblock EMS Textbooks in Mathematics. European Mathematical Society (EMS),
  Z\"urich, 2009.

\bibitem[WW09]{WeiW09}
G.~Wei and W.~Wylie.
\newblock Comparison geometry for the {B}akry-{E}mery {R}icci tensor.
\newblock {\em J. Differential Geom.}, 83(2):377--405, 2009.

\bibitem[Yau75]{Yau75}
S.~T. Yau.
\newblock {Harmonic functions on complete Riemannian manifolds}.
\newblock {\em Comm. Pure Appl. Math.}, 28:201--228, 1975.

\end{thebibliography}
\bibliographystyle{alpha}

\end{document}